\documentclass[11pt,reqno]{amsart}
\usepackage[a4paper, hmargin={2.7cm,2.7cm},vmargin={3.3cm,3.3cm}]{geometry}
\usepackage[english]{babel}
\usepackage{color}
\usepackage[noadjust]{cite}
\usepackage{amsmath}
\usepackage{amssymb}
\usepackage{amsfonts}
\usepackage{amsthm}
\usepackage{enumerate}
\usepackage{amsthm}
\usepackage{dsfont}
\usepackage{todonotes}
\usepackage{etoolbox}

\numberwithin{equation}{section}

\makeatletter
\patchcmd{\ttlh@hang}{\parindent\z@}{\parindent\z@\leavevmode}{}{}
\patchcmd{\ttlh@hang}{\noindent}{}{}{}
\makeatother

\theoremstyle{plain}
\newtheorem{theorem}{Theorem}[section]
\newtheorem{lemma}[theorem]{Lemma}

\theoremstyle{definition}

\newtheorem{examplex}[theorem]{Example}
\newenvironment{example}
  {\pushQED{\qed}\examplex}
  {\popQED\endexamplex}

\theoremstyle{remark}

\def\XXint#1#2#3{{\setbox0=\hbox{$#1{#2#3}{\int}$ }
\vcenter{\hbox{$#2#3$ }}\kern-.6\wd0}}

\DeclareMathOperator{\Span}{\overline{span}}

\DeclareMathOperator{\vol}{vol}

\newcommand{\Hpi}{\mathcal{H}_{\pi}}

\title[Density conditions with stabilisers for lattice orbits]{Density conditions with stabilisers for lattice orbits of discrete series representations}

\author{Jordy Timo van Velthoven}

\address{Faculty of Mathematics,
University of Vienna, 
Oskar-Morgenstern-Platz 1,
1090 Vienna, Austria}
\email{jordy.timo.van.velthoven@univie.ac.at}

\date{}

\begin{document}

\begin{abstract}
This note provides elementary proofs for necessary density conditions for frames and Riesz sequences in the lattice orbit of a discrete series representation that involve the projective stabiliser of the vector. The presented approach extends and simplifies known density conditions for Bergman kernels and lattices in semisimple Lie groups.
\end{abstract}

\maketitle

\section{Introduction}
Let $G$ be a second countable unimodular group with Haar measure $\mu_G$ and let $(\pi, \Hpi)$ be a discrete series representation of $G$ of formal degree $d_{\pi} > 0$. For a vector $g \in \Hpi$ and a lattice $\Gamma \leq G$, the associated lattice orbit is the system
\[
\pi(\Gamma) g = \{ \pi (\gamma) g : \gamma \in \Gamma \}. 
\]
The relation between certain spanning properties of such systems and the lattice covolume $\vol(G/\Gamma)$
has been studied via various approaches in, e.g., \cite{bekka2004square, enstad2025dynamical, romero2022density}. In particular, it is known that if $\pi(\Gamma) g$ forms a \emph{frame} for $\Hpi$, then $\vol(G/\Gamma) d_{\pi} \leq 1$, whereas if $\pi(\Gamma) g$ forms a \emph{Riesz sequence} in $\Hpi$, then $\vol(G/\Gamma) d_{\pi} \geq 1$; see Section \ref{sec:prelim} for precise definitions. These results provide necessary conditions for the existence of frames and Riesz sequences in the orbit of a lattice in terms of the lattice covolume or its reciprocal, the so-called \emph{density} of the lattice. 

For certain classes of representations $(\pi, \Hpi)$ and groups $G$, the above mentioned necessary density conditions are sharp in the sense that for any lattice $\Gamma \leq G$ satisfying $\vol(G/\Gamma) d_{\pi} \leq 1$ (resp.~$\vol(G/\Gamma) d_{\pi} \geq 1$), there exists a vector $g \in \Hpi$ generating a frame (resp.~Riesz sequence) $\pi(\Gamma) g$, see, e.g.,  \cite{bekka2004square,enstad2022sufficient, enstad2022density}. However, these necessary density conditions might not be optimal for a particular fixed vector $g \in \Hpi$ yielding a frame or Riesz sequence $\pi(\Gamma) g$. The general question of finding optimal necessary density conditions for general complete systems $\pi(\Gamma) g$ that depend on the vector $g$ can be traced back to \cite{perelomov1972coherent}. For this, a particularly relevant example  (cf. Example \ref{ex:holomorphic}) is given by the holomorphic discrete series representations $\pi$ of $\mathrm{PSL}(2, \mathbb{R})$, for which the papers \cite{jones2023bergman, kellylyth1999uniform, seip1993beurling, perelomov1974coherent} show that for certain specific vectors $g \in \Hpi$ and lattices $\Gamma \leq \mathrm{PSL}(2, \mathbb{R})$ the completeness (hence, frame property) of $\pi(\Gamma) g$ implies that
\begin{align} \label{eq:stabiliser}
\vol(G/ \Gamma) d_{\pi} \leq |\Gamma_{[g]}|^{-1}, 
\end{align}
where $\Gamma_{[g]} := \{ \gamma \in \Gamma : \pi(\gamma) g \in \mathbb{C} \cdot g \}$ is the projective stabiliser group of $g$; see \cite[Section 9.1]{romero2022density} for a detailed discussion of the papers \cite{jones2023bergman, kellylyth1999uniform, seip1993beurling, perelomov1974coherent}. The necessary density condition \eqref{eq:stabiliser} for frames/complete systems has been extended to general vectors and discrete series representations of semisimple Lie groups in \cite{caspers2022density}. 
In contrast to a frame, a lattice orbit $\pi(\Gamma) g$ can only be a Riesz sequence when the projective stabiliser $\Gamma_{[g]}$ is trivial, since any Riesz sequence is linearly independent. Hence, in general, only systems of the form $\pi(\Lambda) g$, where $\Lambda \subseteq \Gamma$ is a fundamental domain for $\Gamma_{[g]}$, can be expected to be a Riesz sequence. 
For specific Riesz sequences $\pi(\Lambda) g$ in the orbit of $\mathrm{PSL}(2, \mathbb{R})$, it was shown in \cite{kellylyth1999uniform, seip1993beurling} that necessarily
\begin{align} \label{eq:stabiliser2}
\vol(G/ \Gamma) d_{\pi} \geq |\Gamma_{[g]}|^{-1};
\end{align}
see \cite{caspers2022density} for an extension to general semisimple Lie groups. 

The aim of the present note is to provide an elementary approach to the necessary density conditions with stabilisers, both for frames \eqref{eq:stabiliser} and for Riesz sequences \eqref{eq:stabiliser2}. The main result is as follows.

\begin{theorem} \label{thm:intro}
Suppose $\Gamma \leq G$ is a lattice and  $(\pi, \Hpi)$ is a discrete series representation of formal degree $d_{\pi} > 0$. 
Then, for any $g \in \Hpi \setminus \{0\}$, the projective stabiliser \[ \Gamma_{[g]} := \{ \gamma \in \Gamma : \pi(\gamma) g \in \mathbb{C} g \} \] is finite. In addition, given a fundamental domain $\Lambda \subseteq \Gamma$ for $\Gamma_{[g]}$, the following assertions hold:
\begin{enumerate}[(i)]
\item If $\pi(\Gamma) g$ is a frame for $\Hpi$, then
$
\vol(G/\Gamma) d_{\pi} \leq |\Gamma_{[g]}|^{-1}.
$
\item If $\pi(\Lambda) g$ is a Riesz sequence in $\Hpi$, then
$
\vol(G/\Gamma) d_{\pi} \geq |\Gamma_{[g]}|^{-1}.
$
\end{enumerate}
(The quantity $\vol(G/\Gamma) d_{\pi}$ is independent of the choice of Haar measure on $G$.)
\end{theorem}

The proof of Theorem \ref{thm:intro} presented here is entirely self-contained and is based solely on basic theory of frames and Riesz bases. In contrast, the necessary density conditions with stabilisers for lattice orbits of $\mathrm{PSL}(2, \mathbb{R})$ in \cite{kellylyth1999uniform, perelomov1974coherent, seip1993beurling} and \cite{jones2023bergman, caspers2022density} are obtained using advanced techniques from complex analysis/symmetric spaces and von Neumann algebras, respectively. Our method forms a refinement of the elementary proof of the necessary conditions of the density theorem presented in \cite{romero2022density}. 
This method does not seem to give a similarly simple proof of the 
necessary density condition \eqref{eq:stabiliser} for merely complete lattice orbits.

Lastly, Theorem \ref{thm:intro} naturally extends to \emph{projective} discrete series representations (see Theorem \ref{thm:main}). The use of projective representations does not only allow to treat genuine representations that are merely square-integrable modulo their projective kernel, but also allows the treatment of the \emph{projective} holomorphic discrete series representations on weighted Bergman spaces in a direct manner.  The precise notions used in the present note are discussed in Section \ref{sec:prelim}. 

\section{Lattice orbits of projective discrete series representations} \label{sec:prelim}
Let $G$ be a second countable unimodular group with Haar measure $\mu_G$. A projective unitary representation $(\pi, \Hpi)$ of $G$ is a strongly measurable map $\pi : G \to \mathcal{U}(\Hpi)$ satisfying 
\[
\pi(x) \pi(y) = \sigma(x,y) \pi(xy) \quad \text{and} \quad \pi(e) = I_{\Hpi}
\]
for  all $x, y \in G$ and some function $\sigma : G \times G \to \mathbb{T}$, called the \emph{cocycle} of $\pi$.  A projective unitary representation is called \emph{irreducible} if $\{ 0 \}$ and $\Hpi$ are the only closed $\pi$-invariant subspaces of $\Hpi$, and it is called a \emph{projective discrete series representation} if it is irreducible and there exists $g \in \Hpi \setminus \{0\}$ such that
\[
\int_G |\langle g, \pi (x) g \rangle |^2 \; d\mu_G (x) < \infty. 
\]
In this case, there exists a unique $d_{\pi} > 0$, called the \emph{formal degree} of $\pi$, such that
\begin{align} \label{eq:ortho}
\int_G |\langle f, \pi(x) g \rangle |^2 \; d\mu_G (x) = d_{\pi}^{-1} \| f \|^2 \| g \|^2
\end{align}
for all $f, g \in \Hpi$. 

A \emph{fundamental domain} for a discrete subgroup $\Gamma \leq G$ is a Borel measurable set $\Omega \subseteq G$ such that
$
G = \bigcup_{\gamma \in \Gamma} \Omega \gamma$ and $ \Omega \gamma \cap \Omega \gamma' = \emptyset$ for all $\gamma, \gamma' \in \Gamma$ with $\gamma \neq \gamma'$. Every discrete subgroup admits a fundamental domain, 
and a discrete subgroup is called a \emph{lattice} if it admits a fundamental domain of finite measure. Any two fundamental domains have the same measure, and the \emph{covolume} $\vol(G/\Gamma)$ of a lattice $\Gamma$ is defined to be the measure of a fundamental domain. 

Given a discrete set $\Lambda \subseteq G$, a family $\pi(\Lambda) g$ is said to be a \emph{frame} for its span $\Span \pi(\Lambda) g$ if there exist $A, B > 0$, called \emph{frame bounds}, such that
\[
A \| f \|^2 \leq \sum_{\lambda \in \Lambda} |\langle f, \pi(\lambda) g \rangle |^2 \leq B \| f \|^2 \quad \text{for all} \quad f \in \Span \pi(\Lambda) g. 
\]
A system $\pi(\Lambda) g$ satisfying merely the upper frame bound is called a \emph{Bessel sequence} in $\Span \pi(\Lambda) g$, and a frame with frame bounds $A=B=1$ is called a \emph{Parseval frame} for $\Span \pi(\Lambda) g$. 
Equivalently, $\pi(\Lambda) g$ is a Bessel sequence in (resp.~a frame for) $\Span \pi(\Lambda) g$ if its associated \emph{frame operator} 
\[
S_{\Lambda} := \sum_{\lambda \in \Lambda} \langle \cdot , \pi(\lambda) g \rangle \pi(\lambda) g
\]
is bounded (resp.~boundedly invertible) on $\Span \pi(\Lambda) g$. If $\pi(\Lambda) g$ is a frame, then also $S_{\Lambda}^{-1} \pi(\Lambda) g$ (resp. $S_{\Lambda}^{-1/2} \pi(\Lambda) g$)  is a frame (resp. Parseval frame), called the \emph{canonical dual frame} (resp. \emph{canonical Parseval frame}) of $\pi(\Lambda) g$. A system $\pi(\Lambda) g$ is a \emph{Riesz sequence} in $\Hpi$ if there exist $A, B>0$, called \emph{Riesz bounds}, such that
\[
A \| c\|^2 \leq \bigg\| \sum_{\lambda \in \Lambda} c_{\lambda} \pi(\lambda) g \bigg\|^2 \leq B \| c \|^2 \quad \text{for all} \quad c \in \ell^2 (\Lambda). 
\]
Any Riesz sequence is a frame for its span and its canonical dual frame is a biorthogonal system. 

\section{Proof of Theorem \ref{thm:intro}} 
We start by recalling the following basic lemma (cf.~\cite[Proposition 7.2]{romero2022density}), whose short proof we provide for the sake of being self-contained.

\begin{lemma} \label{lem:basic}
Let $(\pi, \Hpi)$ be a projective discrete series representation of formal degree $d_{\pi} > 0$. 
Let $\Gamma \leq G$ be a lattice and let $g \in \Hpi$. If $\pi(\Gamma) g$ is a Bessel sequence in $\Hpi$ with bound $B$, then $d_{\pi}^{-1} \| g \|^2 \leq B \vol(G/\Gamma)$. If $\pi (\Gamma) g$ is also a frame for $\Hpi$ with lower frame bound $A$, then
\[
A \vol(G/\Gamma) \leq d_{\pi}^{-1} \| g \|^2 \leq B \vol(G/\Gamma).
\]
\end{lemma}
\begin{proof}
Let $f \in \Hpi \setminus \{0\}$. If $\Omega \subseteq G$ is a fundamental domain for $\Gamma$, then, using the orthogonality relations \eqref{eq:ortho}, a direct calculation gives
\begin{align*}
d_{\pi}^{-1} \| f \|^2 \| g \|^2 = \sum_{\gamma \in \Gamma} \int_{\Omega} |\langle f, \pi(x \gamma) g \rangle |^2 \; d\mu_G (x) = \int_{\Omega} \sum_{\gamma \in \Gamma} |\langle \pi(x)^* f, \pi(\gamma) g \rangle |^2 \; d\mu_G (x). 
\end{align*}
Hence, using the upper frame bound  $B > 0$ for $\pi(\Gamma) g$ gives 
$
d_{\pi}^{-1} \| f \|^2 \| g \|^2 \leq B \mu_G (\Omega) \| f \|^2,
$
and, if $\pi(\Gamma) g$ also has a lower bound $A > 0$, then also $
d_{\pi}^{-1} \| f \|^2 \| g \|^2 \geq A \mu_G (\Omega) \| f \|^2.
$
\end{proof}

Using Lemma \ref{lem:basic}, we are able to provide a simple proof of Theorem \ref{thm:intro}. 

\begin{theorem} \label{thm:main}
Suppose $\Gamma \leq G$ is a lattice and $(\pi, \Hpi)$ is a projective discrete series representation of formal degree $d_{\pi} > 0$. Then, for any $g \in \Hpi \setminus \{0\}$, the projective stabiliser 
\[ \Gamma_{[g]} := \{ \gamma \in \Gamma : \pi(\gamma) g \in \mathbb{C} g \} \] is finite.
 In addition, if $\Lambda$ is a fundamental domain for $ \Gamma_{[g]}$ in $\Gamma$, then the following assertions hold:
\begin{enumerate}[(i)]
\item If $\pi(\Gamma) g$ is a frame for $\Hpi$, then
$
\vol(G/\Gamma) d_{\pi} \leq |\Gamma_{[g]}|^{-1}.
$
\item If $\pi(\Lambda) g$ is a Riesz sequence in $\Hpi$, then
$
\vol(G/\Gamma) d_{\pi} \geq |\Gamma_{[g]}|^{-1}.
$
\end{enumerate}
(The quantity $\vol(G/\Gamma) d_{\pi}$ is independent of the choice of Haar measure on $G$.)
\end{theorem}

\begin{proof}
Throughout the proof, let $u : \Gamma_{[g]} \to \mathbb{T}$ be the function satisfying $\pi(\gamma) g = u(\gamma) g$ for $\gamma \in \Gamma_{[g]}$. If $\Omega \subseteq G$ is a fundamental domain for $\Gamma_{[g]}$, then the identity \eqref{eq:ortho} yields
\begin{align*}
d_{\pi}^{-1} \| g \|^4 &= \int_G |\langle g, \pi(x) g \rangle |^2 \; d\mu_G (x) = \sum_{\gamma \in \Gamma_{[g]}} \int_{\Omega} |\langle g, \pi(x) u(\gamma) g \rangle |^2 \; d\mu_G (x) \\
&= |\Gamma_{[g]}|  \int_{\Omega} |\langle g, \pi(x) g \rangle |^2 \; d\mu_G (x),
\end{align*}
which shows that $\Gamma_{[g]}$ must be finite. 
Since $\Lambda \subseteq \Gamma$ is a fundamental domain for $\Gamma_{[g]}$ in $\Gamma$, each $\gamma \in \Gamma$ can be written as $\gamma = \lambda \gamma'$ for some unique $\lambda \in \Lambda$ and $\gamma' \in \Gamma_{[g]}$. Therefore, the map $(\lambda, \gamma') \mapsto \lambda \gamma'$ is a bijection from $\Lambda \times \Gamma_{[g]}$ onto $\Gamma$, and 
\[
\Span \pi(\Gamma) g = \Span \{ \overline{\sigma(\lambda, \gamma')} \pi(\lambda) u( \gamma') g : \lambda \in \Lambda, \gamma' \in \Gamma_{[g]} \} = \Span \pi (\Lambda) g .
\] 
In addition, the frame operators $S_{\Lambda}, S_{\Gamma} : \Span \pi(\Lambda) g \to \Span \pi(\Lambda) g$ associated to the systems $\pi(\Lambda) g$ and $\pi(\Gamma) g$ are related by
\begin{align*}
S_{\Gamma} &= \sum_{\gamma \in \Gamma} \langle \cdot , \pi(\gamma) g \rangle \pi(\gamma) g = \sum_{\lambda \in \Lambda} \sum_{\gamma' \in \Gamma_{[g]}} \langle \cdot , \pi(\lambda \gamma') g \rangle \pi(\lambda \gamma') g \\
&= \sum_{\lambda \in \Lambda} \sum_{\gamma' \in \Gamma_{[g]}} |\sigma(\lambda, \gamma')|^2 |u(\gamma')|^2 \langle \cdot, \pi(\lambda) g \rangle \pi(\lambda) g \\
&= |\Gamma_{[g]}| S_{\Lambda}.
\end{align*}
 Hence, $\pi(\Gamma) g$ is a frame for $\Span \pi(\Gamma) g$ if and only if $\pi(\Lambda) g$ is a frame for $\Span \pi(\Lambda) g$. In this case, the associated canonical Parseval frames $S_{\Gamma}^{-1/2} \pi(\Gamma) g$ and $S_{\Lambda}^{-1/2} \pi(\Lambda) g$ satisfy 
 \begin{align} \label{eq:parseval_norms}
\| S_{\Gamma}^{-1/2} \pi(\lambda \gamma') g \|^2 = \| S_{\Gamma}^{-1/2} \pi(\lambda) u(\gamma') g \|^2 = |\Gamma_{[g]}|^{-1} \| S_{\Lambda}^{-1/2} \pi(\lambda) g \|^2
\end{align}
for all $(\lambda, \gamma') \in \Lambda \times \Gamma_{[g]}$. 
Lastly, we note that if $\pi(\Gamma) g$ is a frame for $\Span \pi(\Gamma) g$, then its frame operator $S_{\Gamma}$ satisfies $S_{\Gamma} \pi (\gamma) = \pi(\gamma) S_{\Gamma}$ for all $\gamma \in \Gamma$, and hence also $S^{-1/2}_{\Gamma}$ and $S^{-1}_{\Gamma}$ commute with each operator $\pi(\gamma)$ with $\gamma \in \Gamma$.  
\\~\\
(i) Suppose that $\pi(\Gamma) g$ is a frame for $\Hpi$. Then $ \pi(\Gamma) S^{-1/2}_{\Gamma} g = S^{-1/2}_{\Gamma} \pi(\Gamma) g$ is a frame for $\Hpi$ with frame bounds $A=B=1$, and hence an application of Lemma \ref{lem:basic} gives
\[
 \vol(G/\Gamma) d_{\pi} = \| S^{-1/2}_{\Gamma} g \|^2.
\]
Using Equation \eqref{eq:parseval_norms}  gives
\begin{align*}
 \vol(G/\Gamma) d_{\pi} &= \| \pi(\lambda \gamma') S^{-1/2}_{\Gamma} g \|^2 = \|  S^{-1/2}_{\Gamma} \pi(\lambda \gamma') g \|^2 \\
 &= |\Gamma_{[g]}|^{-1} \| S_{\Lambda}^{-1/2} \pi(\lambda) g \|^2 
 \end{align*}
for $(\lambda, \gamma') \in \Lambda \times \Gamma_{[g]}$. Since $S_{\Lambda}^{-1/2} \pi(\Lambda) g$ is a Parseval frame for $\Hpi$, it follows that
\[
\| S_{\Lambda}^{-1/2} \pi(\lambda) g \|^4 \leq \sum_{\lambda' \in \Lambda} |\langle S_{\Lambda}^{-1/2} \pi(\lambda) g, S_{\Lambda}^{-1/2} \pi(\lambda') g \rangle |^2 = \| S_{\Lambda}^{-1/2} \pi(\lambda) g \|^2,
\]
and thus $\| S_{\Lambda}^{-1/2} \pi(\lambda) g \|^2 \leq 1$ for all $\lambda \in \Lambda$. A combination of the above observations gives 
$
 \vol(G/\Gamma) d_{\pi} \leq |\Gamma_{[g]}|^{-1},
$
as required.
\\~\\
(ii) Suppose that $\pi(\Lambda) g$ is a Riesz sequence in $\Hpi$. Then $\pi(\Lambda) g$ is a Riesz basis for  $\Span \pi(\Lambda) g$. In particular, the frame operator $S_{\Lambda} : \Span \pi (\Lambda) g \to \Span (\Lambda) g$ is invertible, and the canonical dual frame $S^{-1}_{\Lambda} \pi(\Lambda) g$ in $\Span \pi(\Lambda) g$ is a biorthogonal system for $\pi(\Lambda) g$. Therefore,
\begin{align} \label{eq:orthonormal}
\langle S^{-1/2}_{\Lambda} \pi(\lambda) g , S^{-1/2}_{\Lambda} \pi(\lambda') g \rangle = \langle  \pi(\lambda) g , S^{-1}_{\Lambda} \pi(\lambda') g \rangle = \delta_{\lambda, \lambda'}
\end{align}
for $\lambda, \lambda' \in \Lambda$, which shows that $S^{-1/2}_{\Lambda} \pi(\Lambda) g$ is an orthonormal system in $\Hpi$; in particular, it is a Bessel sequence in $\Hpi$ with upper bound $1$. Since $S_{\Gamma} = |\Gamma_{[g]}| S_{\Lambda}$ as operators on $\Span \pi(\Lambda) g = \Span \pi(\Gamma) g$, it follows that also $S_{\Gamma}$ is positive and invertible on $\Span \pi(\Gamma) g$, so that $S_{\Gamma}^{-1/2}$ is well-defined on $\Span \pi(\Gamma) g$. A combination of these observations 
yields, for $f \in \Hpi$, that
\begin{align*}
\sum_{\gamma \in \Gamma} | \langle f, S^{-1/2}_{\Gamma} \pi(\gamma) g \rangle |^2 &= \sum_{\lambda \in \Lambda} \sum_{\gamma' \in \Gamma_{[g]}} | \langle f, S_{\Gamma}^{-1/2} \pi(\lambda) u(\gamma') g \rangle |^2 \\
&= \sum_{\lambda \in \Lambda} \sum_{\gamma' \in \Gamma_{[g]}}  |\langle f, |\Gamma_{[g]}|^{-1/2} S_{\Lambda}^{-1/2} \pi(\lambda) g \rangle |^2 \\
&= \sum_{\lambda \in \Lambda}  |\langle f, S_{\Lambda}^{-1/2} \pi(\lambda) g \rangle |^2 \\
&\leq \| f \|^2.
\end{align*} 
Hence, the system $ \pi(\Gamma) S^{-1/2}_{\Gamma} g = S^{-1/2}_{\Gamma} \pi(\Gamma) g$ is also a Bessel sequence in $\Hpi$ with Bessel bound $1$. An application of Lemma \ref{lem:basic} thus gives
\[
 \vol(G/\Gamma) d_{\pi} \geq \| S^{-1/2}_{\Gamma} g \|^2.
\]
Combining this with the identity \eqref{eq:parseval_norms} yields, for $(\lambda, \gamma') \in \Lambda \times \Gamma_{[g]}$, 
\begin{align*}
 \vol(G/\Gamma) d_{\pi} \geq \| \pi(\lambda \gamma') S^{-1/2}_{\Gamma} g \|^2 = |\Gamma_{[g]}|^{-1} \| S^{-1/2}_{\Lambda} \pi(\lambda) g \|^2 = |\Gamma_{[g]}|^{-1},
\end{align*}
where the last equality used the identity \eqref{eq:orthonormal}.
\end{proof}

The following example demonstrates Theorem \ref{thm:main} for holomorphic discrete series representations of $\mathrm{PSL}(2, \mathbb{R})$. This setting exemplifies the role of the projective stabilisers in Theorem \ref{thm:main} and allows to relate them to stabilisers of the action of the group on the upper half-plane. See the papers \cite{jones2023bergman, kellylyth1999uniform, seip1993beurling, perelomov1974coherent} and the overview \cite[Section 9.1]{romero2022density} for the related original results in this setting.

\begin{example} \label{ex:holomorphic}
The group $G = \mathrm{PSL}(2, \mathbb{R})$ acts on the upper half-plane $\mathbb{C}^+$ through the action
\[
G \times \mathbb{C}^+ \ni \big( 
\begin{pmatrix}
a & b \\
c & d
\end{pmatrix}
, z) 
\mapsto \frac{az+b}{cz+d} \in \mathbb{C}^+.
\]
A $G$-invariant measure on $\mathbb{C}^+$ is given by $d\mu(z) = (\Im(z))^{-1} dz$, where $dz$ is the Lebesgue measure on $\mathbb{C}^+$. For $\alpha > 1$, define the weighted measure $d\mu_{\alpha}(z) = (\Im(z))^{\alpha-1} d\mu(z)$ and the associated weighted Bergman space $A^2_{\alpha} (\mathbb{C}^+)$ as the space of holomorphic functions $f : \mathbb{C}^+ \to \mathbb{C}$ satisfying
\[
\| f \|^2_{A^2_{\alpha}} := \int_{\mathbb{C}^+} |f(z)|^2 \; d\mu_{\alpha}(z) < \infty. 
\]
The space $A^2_{\alpha} (\mathbb{C}^+)$ is a reproducing kernel Hilbert space with respect to the norm $\| \cdot \|_{A^2_{\alpha}}$. 
For $\alpha >1$, the \emph{holomorphic discrete series representation} $\pi_{\alpha}$ of $G$ on $A^2_{\alpha} (\mathbb{C}^+)$ is given by the projective unitary representation defined by
\[
\pi_{\alpha} (\mathrm{m}) f(z) = j(\mathrm{m}^{-1}, z)^{\alpha} f(\mathrm{m}^{-1} \cdot z), \quad \mathrm{m} = 
\begin{pmatrix} 
a & b \\
c & d
\end{pmatrix} \in G, \; z \in \mathbb{C}^+,
\]
where $j(m, z) := (cz+d)^{-1}$. 
Each representation $\pi_{\alpha}$ is irreducible and square-integrable.

Let $k_z^{(\alpha)} (w) = 2^{\alpha - 2} \pi^{-1} (\alpha - 1) i^{\alpha} (w - z)^{-\alpha}$ denote the reproducing kernel function in $ A^2_{\alpha} (\mathbb{C}^+)$ at $z \in \mathbb{C}^+$, so that
\[
f(z) = \langle f, k_z^{(\alpha)} \rangle, \quad f \in A^2_{\alpha} (\mathbb{C}^+), \; z \in \mathbb{C}^+.
\]
If $\sigma_{\alpha} : G \times G \to \mathbb{T}$ denotes the cocycle of $\pi_{\alpha}$, then a direct calculation gives 
\begin{align*}
\langle f, \pi_{\alpha} (\mathrm{m}) k_z^{(\alpha)} \rangle =  
\overline{\sigma(\mathrm{m}, \mathrm{m}^{-1})} \big(\pi_{\alpha} (\mathrm{m}^{-1}) f\big)(z) 
= \overline{\sigma(\mathrm{m}, \mathrm{m}^{-1})} j(\mathrm{m}, z)^{\alpha} \langle f, k_{\mathrm{m} \cdot z}^{(\alpha)} \rangle, \quad \mathrm{m} \in G, \; z \in \mathbb{C}^+,
\end{align*}
for any $f \in A^2 (\mathbb{C}^+)$, so that 
\begin{align} \label{eq:action_rkhs}
\pi_{\alpha}(m) k_z^{(\alpha)} = 
\sigma(\mathrm{m}, \mathrm{m}^{-1}) \overline{j(\mathrm{m}, z)^{\alpha} }k_{\mathrm{m} \cdot z}^{(\alpha)}, \quad \mathrm{m} \in G, \; z \in \mathbb{C}^+. 
\end{align}
Therefore, if $\Gamma \leq G$ is a lattice, then the stabiliser subgroup of a point $z \in \mathbb{C}^+$ satisfies
\[
\Gamma_z := \{ \gamma \in \Gamma : \gamma \cdot z = z \} \subseteq \Gamma_{[k_z^{\alpha}]}. 
\]
For the reverse inclusion, let $u : \Gamma_{[k_z^{\alpha}]} \to \mathbb{T}$ be such that $\pi_{\alpha} (\gamma) k_z^{\alpha} = u(\gamma) k_z^{\alpha}$ for all $\gamma \in \Gamma_{[k_z^{\alpha}]}$. If $\gamma \in \Gamma_{[k_z^{\alpha}]}$, then \eqref{eq:action_rkhs} yields
\[
u(\gamma) k_z^{(\alpha)} = \sigma(\gamma, \gamma^{-1}) \overline{j(\gamma,z)^{\alpha}} k^{(\alpha)}_{\gamma \cdot z}
\]
which implies that $\gamma \cdot z = z$ by analyticity, and thus $\Gamma_z = \Gamma_{[k_z^{\alpha}]}$. 
Therefore, if $\pi_{\alpha} (\Gamma) k_z^{(\alpha)}$ forms a frame for $A^2_{\alpha} (\mathbb{C}^+)$, then
$
\vol(G/\Gamma) d_{\pi_{\alpha}} \leq |\Gamma_z|^{-1}.
$
Similarly, if $\Lambda \subseteq \Gamma$ is a fundamental domain for $\Gamma_z$ and if $\pi_{\alpha} (\Lambda) k_z^{(\alpha)}$ is a Riesz sequence in $A^2_{\alpha} (\mathbb{C}^+)$, then 
$
\vol(G/\Gamma) d_{\pi_{\alpha}} \geq |\Gamma_z|^{-1}
$. Cf. Theorem \ref{thm:main}. 
\end{example}

\section*{Acknowledgements}
This research was funded in whole or in part by the Austrian Science Fund (FWF):
10.55776/PAT2545623.

\bibliographystyle{abbrv}
\bibliography{bib}

\end{document}